\documentclass[10pt, a4paper, reqno, oneside]{amsart}

\usepackage{amsmath, amsopn, amsfonts, amsthm, amssymb, amscd, enumerate, multicol, scalefnt}
\usepackage{sansmath}
\usepackage{etex}
\usepackage{hyperref}
\usepackage{t1enc}
\usepackage[latin1]{inputenc}
\usepackage{graphicx}
\usepackage[all]{xy}
\usepackage{color}
\usepackage{slashed}
\usepackage[mathscr]{euscript}
\usepackage{enumitem}
\setlist[itemize]{noitemsep, topsep=1pt, leftmargin=30pt}
\newcommand\bcdot{\ensuremath{
  \mathchoice
   {\mskip\thinmuskip\lower0.2ex\hbox{\scalebox{1.6}{$\cdot$}}\mskip\thinmuskip}}
   {\mskip\thinmuskip\lower0.2ex\hbox{\scalebox{1.6}{$\cdot$}}\mskip\thinmuskip}
   {\lower0.3ex\hbox{\scalebox{1.2}{$\cdot$}}}
   {\lower0.3ex\hbox{\scalebox{1.2}{$\cdot$}}}
}
\theoremstyle{plain}
\newtheorem{theo}{Theorem}[section]

\newtheorem{prop}[theo]{Proposition}
\newtheorem{fact}[theo]{Fact}

\theoremstyle{definition}

\newtheorem{definition}[theo]{Definition}

\theoremstyle{plain}
\newtheorem{lemma}[theo]{Lemma}
\newtheorem{theorem}[theo]{Theorem}
\newtheorem{corollary}[theo]{Corollary}

\theoremstyle{definition}

\newtheorem{remark}[theo]{Remark}

\theoremstyle{plain}
\newtheorem{thmint}{Theorem}

\newtheorem{propint}[thmint]{Proposition}

\theoremstyle{definition}

\renewcommand{\=}{:=}

\renewcommand{\a}{\alpha}

\renewcommand{\d}{\delta}
\newcommand{\e}{\varepsilon}
\newcommand{\f}{\varphi}

\newcommand{\w}{\omega}

\newcommand{\s}{\sigma}

\renewcommand{\t}{\tau}

\newcommand{\D}{\Delta}

\renewcommand{\L}{\Lambda}
\newcommand{\Q}{\Theta}

\newcommand{\W}{\Omega}

\newcommand{\bR}{\mathbb{R}}
\newcommand{\bZ}{\mathbb{Z}}

\newcommand{\bN}{\mathbb{N}}

\newcommand{\fG}{\mathsf{G}}

\newcommand{\fSL}{\mathsf{SL}}

\newcommand{\fSU}{\mathsf{SU}}

\newcommand{\gm}{\mathfrak{m}}

\newcommand{\so}{\mathfrak{so}}

\newcommand{\cC}{\mathcal{C}}

\newcommand{\cU}{\mathcal{U}}
\newcommand{\cV}{\mathcal{V}}

\newcommand{\eB}{\EuScript{B}}

\newcommand{\eD}{\EuScript{D}}

\newcommand{\eU}{\EuScript{U}}

\newcommand{\eW}{\EuScript{W}}

\newcommand{\st}{{\operatorname{st}}}

\newcommand{\p}{\partial}

\newcommand{\rar}{\rightarrow}

\newcommand{\la}{\langle}
\newcommand{\ra}{\rangle}

\renewcommand{\square}{\kern1pt\vbox
{\hrule height 0.6pt\hbox{\vrule width 0.6pt\hskip 3pt \vbox{\vskip
6pt}\hskip 3pt\vrule width 0.6pt}\hrule height0.6pt}\kern1pt}
\renewcommand{\=}{\  \raisebox{0.15mm}{:} {=} \ }
\newcommand{\td}{\mathtt{d}}

\DeclareMathOperator\Tr{Tr}

\DeclareMathOperator\Id{Id}
\DeclareMathOperator\Dom{Dom}

\DeclareMathOperator{\vspan}{span}

\DeclareMathOperator\Exp{\operatorname{Exp}}
\newcommand{\tdist}{\mathtt{d}\hspace{-0.5pt}\mathtt{i}\hspace{-0.5pt}\mathtt{s}\hspace{-0.5pt}\mathtt{t}}

\renewcommand\Im{\operatorname{Im}}

\newcommand\Ric{\operatorname{Ric}}
\newcommand\Rm{\operatorname{Rm}}

\newcommand{\inj}{\operatorname{inj}}

\newcommand{\ol}{\overline}

\newcommand{\zero}{\operatorname{o}}
\def\<#1,#2>{\langle\,#1,\,#2\,\rangle}

\newcommand{\aac}{\`a}
\newcommand{\Aac}{\`A}

\newcommand{\Math}{{\it Mathematica\raise5 pt\hbox{$\scriptscriptstyle \circledR$}7}}
\newcommand{\n}{\nabla}

\newcommand{\beq}{\begin{equation}}
\newcommand{\eeq}{\end{equation}}
\def\<#1,#2>{\langle\,#1,\,#2\,\rangle}
\newcommand{\arr}{\begin{array}{rlll}}
\newcommand{\ea}{\end{array}}
\newcommand{\bea}{\begin{eqnarray}}
\newcommand{\eea}{\end{eqnarray}}
\newcommand{\bean}{\begin{eqnarray*}}
\newcommand{\eean}{\end{eqnarray*}}

\def\sideremark#1{\ifvmode\leavevmode\fi\vadjust{
\vbox to0pt{\hbox to 0pt{\hskip\hsize\hskip1em
\vbox{\hsize3cm\tiny\raggedright\pretolerance10000
\noindent #1\hfill}\hss}\vbox to8pt{\vfil}\vss}}}

\newcounter{ssig}
\newcounter{ttig}
\title[A compactness theorem for locally homogeneous spaces]{A compactness theorem for locally homogeneous spaces}
\author{Francesco Pediconi}

\thanks{This work was supported by project PRIN 2017 ``Real and Complex Manifolds: Topology, Geometry and holomorphic dynamics'' (code 2017JZ2SW5) and by GNSAGA of INdAM} 

\begin{document}
\begin{abstract} We prove the existence and uniqueness of geometric models of local isometry classes of locally homogeneous spaces with sectional curvature $|\sec|\leq 1$. Moreover, we show that the set of geometric models is compact in the pointed $\cC^{1,\a}$-topology. \end{abstract}

\maketitle


\section{Introduction} \setcounter{equation} 0

The Gromov Precompactness Theorem states that the space of compact Riemannian manifolds with uniform lower Ricci curvature bound and uniform upper diameter bound is precompact in the Gromov-Haussdorf topology (see \cite{Gro}). Moreover, even though no reasonable description is known so far for the closure of this class (see \cite[Rem 10.7.5]{BBI}), under some suitable assumptions there exist deep theorems concerning the structure of the limit spaces (see e.g. \cite{CC1,CC2}). One of its many applications is an alternative proof of the Cheeger Finiteness Theorem, stating that the class of compact Riemannian manifolds with uniform upper and lower sectional curvature bounds, uniform lower volume bound and uniform upper diameter bound contains only finitely many diffeomorphism types (see \cite{Che}). \smallskip

In this paper we will state a compactness theorem for locally homogeneous Riemannian spaces, that are Riemannian manifolds $(M,g)$ on which the pseudogroup of local isometries acts transitively, i.e. for any $x,y \in M$ there exists a local isometry of $(M,g)$ mapping $x$ to $y$. Let us stress that any globally homogeneous space is, in particular, locally homogeneous. However, in general, locally homogeneous spaces are incomplete and not locally isometric to any globally homogeneous space (see \cite{Kow,Mos}). We mention e.g. the quotient $(\fSU(2){\times}\fSU(2)) / \D S^1_{\a}$, where the circle $S^1$ is irrationally diagonally immersed in $\fSU(2){\times}\fSU(2)$. \smallskip

Locally homogeneous spaces naturally occur as limits of globally homogeneous spaces. We recall that, by the Cheeger-Gromov Compactness Theorem, a non-collapsing sequence of globally homogeneous spaces with bounded geometry subconverges to a globally homogeneous limit space in the pointed $\cC^{\infty}$-topology (see e.g. \cite[Thm 2.3]{Ha}). Here, a sequence of globally homogeneous spaces $(M^{(n)},g^{(n)})$ is said to have {\it bounded geometry} if for any integer $k\geq0$ there exists $C_k>0$ such that $$\big|\Rm(g^{(n)})\big|_{g^{(n)}}+\big|\n^{g^{(n)}}\!\Rm(g^{(n)})\big|_{g^{(n)}}+\,{\dots}\,+\big|\big(\n^{g^{(n)}}\big)^{\!k}\!\Rm(g^{(n)})\big|_{g^{(n)}} \leq C_k \quad \text{ for any $n \in \bN$ .} $$

However, this framework excludes the study of collapsing sequences with bounded curvature (see \cite{CG1,CG2}). This happens for instance with the Berger metrics on the $3$-sphere, which are obtained from the round metric by shrinking the fibers of the Hopf fibration (see \cite[p. 252]{Bes}). In order to overcome this issue, Glickenstein \cite{Gl} and Lott \cite{Lo} extended the Cheeger-Gromov Compactness Theorem to those sequences of manifolds with no uniform positive lower bound on the injectivity radii. Their key tool is to replace the Riemannian manifolds with the larger category of {\it Riemannian groupoids} (see \cite[Sec 5]{Lo} and references therein). Remarkably, applying this machinery to sequences of globally homogeneous spaces, the limit Riemannian groupoids one gets are incomplete locally homogeneous spaces (see \cite[Ex 5.7, Prop 5.9]{Lo} and \cite[Sec 5]{BLS}). \smallskip

Motivated by \cite{BLS}, we give the following

\vskip 5pt

\noindent \textbf{Definition {\rm (Geometric model)}.} A {\it geometric model} is a smooth locally homogeneous Riemannian distance ball $(\eB, \hat{g})=(\eB_{\hat{g}}(o,\pi), \hat{g})$ of radius $\pi$ satisfying $|\sec(\hat{g})|\leq 1$ and $\inj_o(\eB, \hat{g})=\pi$. If $(M,g)$ is a locally homogeneous space and $(\eB, \hat{g})$ is a geometric model locally isometric to $(M,g)$, then $(\eB, \hat{g})$ is said to be a {\it geometric model of $(M,g)$}.

\vskip 5pt

The first main result of this paper is the following

\begin{thmint}[Existence and uniqueness] Any locally homogeneous space $(M,g)$ with $|\sec(g)|\leq1$ has a geometric model $(\eB,\hat{g})$, which is unique up to isometry. \label{MAIN-A} \end{thmint}

Notice that for a given globally homogeneous space $(M,g)$ with $|\sec(g)|\leq1$, the existence of a geometric model locally isometric to $(M,g)$ is a direct consequence of the Rauch Comparison Theorem, which allows to pull-back the metric $g$ to the tangent ball of radius $\pi$ at a point $p \in M$ (see Remark \ref{remgmHstar}). However, since locally homogeneous space are generally incomplete, the existence of geometric models predicted in Theorem \ref{MAIN-A} is definitely non trivial. \smallskip

The second main result of this paper is the following

\begin{thmint}[Compactness] The space of geometric models is compact in the pointed $\cC^{1,\a}$-topology. \label{MAIN-B} \end{thmint}

We would like to point out that our Theorem \ref{MAIN-B} was already known in a very special case. More precisely, B\"ohm, Lafuente and Simon proved the following statement in \cite[Thm 1.6]{BLS}: {\it a sequence of geometric mo\-dels $\big(\eB^{(n)}, \hat{g}^{(n)}\big)$ with $\Ric(\hat{g}^{(n)}) \rar 0$ converges, up to a subsequence, to a smooth flat Riemannian manifold $\big(M^{(\infty)},g^{(\infty)}\big)$ in the pointed $\cC^{1,\a}$-topology}. Indeed, both theorems rely on a Cheeger-Gromov-type precompactness theorem for incomplete Riemannian manifolds and in both proofs there is the need of showing that the limit space, which is a priori just a $\cC^{1,\a}$-Riemannian manifold, is indeed smooth. In \cite{BLS} such required regularity is achieved by means of the additional assumption on the Ricci tensor, while in our result we exploit a Lie-theoretical argument based on a local version of the Myers-Steenrod Theorem proved in \cite{Ped2}. \smallskip

Notice also that Theorem \ref{MAIN-B} is not true anymore if one replaces the $\cC^{1,\a}$-topology with the $\cC^{\infty}$-topology. In fact, there exist sequences of geometric models $\big(\eB^{(n)}, \hat{g}^{(n)}\big)$ which converges to some limit geometric model $\big(\eB^{(\infty)}, \hat{g}^{(\infty)}\big)$ in the pointed $\cC^{1,\a}$-topology but for which the corresponding sequence of covariant derivatives ${\n}^{\hat{g}^{(n)}}\Rm(\hat{g}^{(n)}) \big|_{o^{(n)}}$ blows up. This phenomenon is somehow unexpected to us and it is related to the existence of diverging sequences of invariant metrics with bounded curvature in the sense of \cite{Ped1}. We discuss it in detail in \cite{Ped3}. \smallskip

Together with Theorem \ref{MAIN-A}, this provides the compactness of the moduli space of locally homogeneous spaces with $|\sec|\leq1$ up to local isometry. Therefore, Theorem \ref{MAIN-A} and Theorem \ref{MAIN-B} show that the geometric models provide the right theoretical framework to study convergence of globally and locally homogeneous spaces from the geometric viewpoint, which is the topic of \cite{Ped3}. \smallskip

In order to use the compactness result of Theorem \ref{MAIN-B} in the proof of Theorem \ref{MAIN-A}, one needs to take care of some special issues on the convergence of locally homogeneous spaces. In particular, we use the fact that {\it the homogenous spaces are dense in the class of locally homogenous spaces with respect to the topology of algebraic convergence}. Here, with ``algebraic convergence'' we mean the convergence of the corresponding full Lie algebra of Killing vector fields at some distinguished points. This is a key feature on equivariant convergence of globally and locally homogeneous spaces: see \cite{Lau1,Lau2,Lau3} and \cite{BL}. As a by-product of our analysis, we obtain the following

\begin{propint} Let $(\eB^{(n)},\hat{g}^{(n)})$ be a sequence of geometric models converging algebraically to a locally homogeneous limit space $(M^{(\infty)},g^{(\infty)})$. Then, $\big|\sec(g^{(\infty)})\big|\leq1$ and $(\eB^{(n)},\hat{g}^{(n)})$ converges to the geometric model $(\eB^{(\infty)},\hat{g}^{(\infty)})$ of $(M^{(\infty)},g^{(\infty)})$ in the pointed $\cC^{\infty}$-topology. \label{MAIN-C} \end{propint}

Here, the role of the geometric models is crucial. In fact, Lauret exhibited an explicit sequence of Aloff--Wallach spaces $\big(W_{n,n{+}1},g^{(n)}\big)$ converging algebraically to a limit Aloff--Wallach space $\big(W_{1,1},g^{(\infty)}\big)$ (see \cite[Ex 6.6]{Lau1}). Since $W_{1,1}$ is compact and $W_{n,n{+}1}$ are pairwise non-homeomorphic, it follows that there is no subsequence of $\big(W_{n,n{+}1},g^{(n)}\big)$ converging to $\big(W_{1,1},g^{(\infty)}\big)$ in the pointed $\cC^{\infty}$-topology. Notice that this implies, in particular, that the injectivity radii along the sequence must tend to zero, otherwise one would have sub-convergence by the Cheeger-Gromov Compactness Theorem. However, up to a rescaling we can assume that $|\sec(g^{(n)})|\leq1$ and, by Proposition \ref{MAIN-C}, the geometric models of $\big(W_{n,n{+}1},g^{(n)}\big)$ converge to the geometric model of $\big(W_{1,1},g^{(\infty)}\big)$ in the pointed $\cC^{\infty}$-topology.

We also stress that the converse of Proposition \ref{MAIN-C} is wrong, i.e. one can build sequences of geometric models converging in the pointed $\cC^{\infty}$-topology which do not admit any algebraically convergent subsequence. This phenomenon is called {\it algebraic collapse} and it happens, for example, by considering sequences of left-invariant Ricci flow blow-downs on the universal cover of $\fSL(2,\bR)$ (see \cite[Ex 9.1]{BL}). \smallskip

The paper is structured as follows. In Section \ref{prel} we collect some preliminaries on locally homogeneous spaces, on local groups of isometries and on convergence of Riemannian manifolds. In Section \ref{spaceGM} we prove Theorem \ref{MAIN-B}. In Section \ref{exsec} we briefly introduce the notion of algebraic convergence and we prove Theorem \ref{MAIN-A} and Proposition \ref{MAIN-C}. \medskip

\noindent{\it Acknowledgement.} This work has been set up during the visit of the author at WWU M\"unster. We are grateful to Christoph B\"ohm for his hospitality and for many fundamental discussions about the contents of this paper. We also thank Ramiro Lafuente, Andrea Spiro and Luigi Verdiani for helpful comments and pleasant conversations. Finally, we would like to thank the anonymous referees for their careful reading of the manuscript.

\vspace{15pt}

\section{Preliminaries and notation} \label{prel} \setcounter{equation} 0

\subsection{Notation} \label{notation} \hfill \par

We indicate with $\la \cdot, \cdot \ra_{\st}$ the standard Euclidean metric on $\bR^m$ and with $|\,\cdot\,|_{\st}$ the induced norm. Given a ball $B \subset\subset \bR^m$ and a pair $(k,\a) \in \big(\bZ_{\geq0} \times [0,1]\big)$, we consider the Banach spaces $\cC^{k,\a}(\bar{B})$ and $\cC^{k,0}(\bar{B})\=\cC^{k}(\bar{B})$ defined following \cite[p. 52]{GT} with the usual norm $|\!| \cdot |\!|_{\cC^{k,\a}(\bar{B})}$. A function $f: U \subset \bR^m \rar \bR$ is said to be {\it of class $\cC^{k,\a}$} if $f|_B \in \cC^{k,\a}(\bar{B})$ for any ball $B \subset\subset U$. In what follows, {\it smooth} will always be a synonym for {\it of class $\cC^{\infty}$}. A sequence $f^{(n)}: U^{(n)} \subset \bR^m \rar \bR$ of functions of class $\cC^{k,\a}$ {\it converges in the $\cC^{k,\a}$-topology} to a function $f^{(\infty)}: U^{(\infty)} \subset \bR^m \rar \bR$ of class $\cC^{k,\a}$ if $\lim_{n\rightarrow +\infty} U^{(n)}=U^{(\infty)}$ and for any ball $B \subset\subset U^{(\infty)}$ it holds \beq \big|\!\big|f^{(n)}-f^{(\infty)}\big|\!\big|_{\cC^{k,\a}(\bar{B})} \rar 0 \quad \text{ as $n \rar +\infty$ } \,\, . \label{convCka} \eeq If in addition $f^{(n)},f^{(\infty)}$ are of class $\cC^{\infty}$ and \eqref{convCka} holds for any $(k,0) \in \bZ_{\geq0}\times\{0\}$, we say that $f^{(n)}$ converges in the $\cC^{\infty}$-topology to $f^{(\infty)}$. Here, we recall that by definition $\lim_{n\rightarrow +\infty} U^{(n)}=U^{(\infty)}$ if both the sets $$\liminf_{n\rightarrow +\infty} U^{(n)} \= \bigcup_{n \in \bN} \bigcap_{j \geq n} U^{(j)} \,\, , \quad \limsup_{n\rightarrow +\infty} U^{(n)} \=  \bigcap_{n \in \bN} \bigcup_{j \geq n} U^{(j)} $$ coincide with $U^{(\infty)}$. \smallskip

In dealing with differentiable manifolds, by \cite[Thm. 2.9]{Hi} it is sufficient to consider only those whose local coordinates overlap smoothly, i.e. smooth manifolds. One can easily work with local coordinates to define functions of class $\cC^{k,\a}$ between smooth manifolds as well as tensor fields of class $\cC^{k,\a}$ on smooth manifolds. A {\it $\cC^{k,\a}$-Riemannian manifold $(M,g)$} is the datum of a smooth manifold $M$ together with a Riemannian metric $g$ on $M$ of class $\cC^{k,\a}$. We recall that any $\cC^{k,\a}$-Riemannian manifold $(M,g)$ is in particular a separable, locally compact length space by means of the Riemannian distance $\td_g$ induced by $g$ on $M$ (see e.g. \cite{Bu}). We will denote by $\eB_g(x,r)$ the metric ball centered at $x \in M$ of radius $r>0$ inside $(M,\td_g)$. If $k+\a>0$, then a map $f:M \rar M$ is an isometry of the metric space $(M,\td_g)$ if and only if it is a $\cC^{k+1,\a}$-diffeomorphism which preserves the metric tensor $g$ (see \cite{CH, Re, Sab, Sh}). \smallskip

Given a sufficiently regular $\cC^{k,\a}$-Riemannian manifold $(M,g)$, we denote by $\n^g$ its Levi-Civita covariant derivative and by $\Rm(g)(X{\wedge}Y) \= \n^g_{[X,Y]}-[\n^g_X,\n^g_Y]$ its Riemannian curvature operator. Moreover, we denote the sectional curvature by $\sec(g)$, the Riemannian exponential map by $\Exp(g)$, the injectivity radius at a distinguished point $x \in M$ by $\inj_x(M,g)$.

\subsection{Pointed convergence} \hfill \par

In this section, we give the notion of pointed convergence for pointed Riemannian manifolds. This definition is usually stated for complete Riemannian manifolds (see e.g. \cite[Ch 3]{Cho}), but here we present a version which holds also in the non-complete setting. \smallskip

Let $M$ be a smooth manifold. We recall that a sequence $(T^{(n)})$ of tensor fields on $M$ of class $\cC^{k,\a}$ {\it converges in the $\cC^{k,\a}$-topology} to a tensor field $T$ on $M$ of class $\cC^{k,\a}$ of the same type, if for any local chart $(U,\xi)$ on $M$, the components of $(\xi^{-1})^*T^{(n)}$ converge to the components of $(\xi^{-1})^*T$ in the $\cC^{k,\a}$-topology. \smallskip

Let now $(M,g)$ be a $\cC^{k,\a}$-Riemannian manifold. The {\it distance of a point $p \in M$ from the boundary of $(M,g)$} is defined as the supremum 
$$\tdist(p,\p(M,g)) \= \sup\{r\geq0 : \eB_g(p,r) \subset\subset M \} \,\, .$$ Notice that by the Hopf-Rinow-Cohn-Vossen Theorem the manifold $(M,g)$ is complete if and only if $\tdist(x,\p(M,g))=+\infty$ for some, and hence for any, $x \in M$ (see e.g. \cite[Thm 2.5.28]{BBI}).

\begin{definition} Let $(M^{(n)},g^{(n)},p^{(n)})$ be a sequence of pointed $\cC^{k,\a}$-Riemannian $m$-manifolds, let $\d^{(n)}\=\tdist(p^{(n)},\p(M^{(n)},g^{(n)}))$ and assume that $\d^{(n)} \rar \d^{(\infty)} \in (0,+\infty]$ . The sequence $(M^{(n)},g^{(n)},p^{(n)})$ is said to {\it converge in the pointed $\cC^{k,\a}$-topology} to a pointed $\cC^{k,\a}$-Riemannian manifold $(M^{(\infty)},g^{(\infty)},p^{(\infty)})$ if there exist: \begin{itemize}
\item[i)] an exhaustion $(U^{(n)})$ of $M^{(\infty)}$ by relatively compact open sets centered at $p^{(\infty)}$,
\item[ii)] a sequence $(\tilde{\d}^{(n)}) \subset \bR$ such that $0<\tilde{\d}^{(n)} \leq \d^{(n)}$ and $\tilde{\d}^{(n)} \rar \d^{(\infty)}$;
\item[iii)] a sequence of $\cC^{k+1,\a}$-embeddings $\phi^{(n)}: U^{(n)} \rar M^{(n)}$ such that $\phi^{(n)}(p^{(\infty)})=p^{(n)}$;
\end{itemize} in such a way that the following conditions are satisfied: \begin{itemize}
\item[$\bcdot$] $\eB_{g^{(n)}}(p^{(n)},\tilde{\d}^{(n)}) \subset \phi^{(n)}(U^{(n)}) \subset \eB_{g^{(n)}}(p^{(n)},\d^{(n)})$;
\item[$\bcdot$] $\phi^{(n)*}g^{(n)}$ converges to $g^{(\infty)}$ in the $\cC^{k,\a}$-topology. 
\end{itemize} \label{defpointed} \end{definition}

Notice that this type of convergence implies that for any $0<r<\d^{(\infty)}$ there exists an integer $\bar{n}=\bar{n}(r) \in \bN$ such that the Riemannian distance ball $\eB_{g^{(n)}}(p^{(n)},r)$ is compactly contained in $M^{(n)}$ for any $n \geq \bar{n}$ and the sequence of compact balls $\big(\overline{\eB_{g^{(n)}}(p^{(n)},r)},\td_{g^{(n)}}\big)_{n \geq \bar{n}}$ converges in the Gromov-Hausdorff topology to the compact ball $\big(\overline{\eB_{g^{(\infty)}}(p^{(\infty)},r)},\td_{g^{(\infty)}}\big)$ (see \cite[p. 415]{Pet} and \cite[Ex 8.1.3]{BBI}). Moreover: \begin{itemize}
\item[$\bcdot$] if $\d^{(\infty)}$ is finite, then the limit space $(M^{(\infty)},g^{(\infty)},p^{(\infty)})$ is an incomplete Riemannian distance ball of radius $\d^{(\infty)}$ centered at $p^{(\infty)}$ and also $\tdist(p^{(\infty)},\p(M^{(\infty)},g^{(\infty)}))=\d^{(\infty)}$;
\item[$\bcdot$] if $\d^{(\infty)}=+\infty$, then the limit space $(M^{(\infty)},g^{(\infty)},p^{(\infty)})$ is complete.
\end{itemize} Notice also that if $\d^{(n)}=+\infty$ for any $n \in \bN$, then we get back to the usual definition of pointed convergence for complete Riemannian manifolds (see e.g. \cite[p. 415]{Pet}). \smallskip

We mention here that the Cheeger-Gromov Precompactness Theorem states that given two constants $D_{\zero},v_{\zero}>0$, the space of smooth compact Riemannian $m$-manifolds with bounded curvature, diameter at most $D_{\zero}$ and volume at least $v_{\zero}$ is precompact in the pointed $\cC^{1,\a}$-topology. Various versions of this classical result are known, e.g. for complete non-compact but pointed Riemannian manifolds or for bounded domains in possibly incomplete pointed Riemannian manifolds. We refer to the survey paper \cite{HH} and references therein for more details. \smallskip

Finally, we recall that a central role in convergence theory of Riemannian manifolds is played by the notion of {\it harmonic radius}, that is

\begin{definition}[\cite{And}, Sec 2] Let $(M^m,g)$ be a smooth Riemannian manifold, $p \in M$. The {\it $\cC^{k,a}$-harmonic radius of $(M^m,g)$ at $p$}, which we denote by ${\rm har}^{k,\a}_p(M,g)$, is the largest $r \geq 0$ such that there exists a local chart $\xi=(\xi^1,{\dots},\xi^m) : \eB_g\big(p,\sqrt{2}r\big) \rar \bR^m$ such that \begin{itemize}
\item[i)] $\xi(p)=0$ and $((\xi^{-1})^*g)_{ij}(0)=\d_{ij}$ for any $1 \leq i,j \leq m$;
\item[ii)] $\D_g \xi^i=0$ for any $1 \leq i \leq m$;
\item[iii)] $\frac12|v|_{\st}^2 \leq ((\xi^{-1})^*g)_{ij}(x) v^iv^j \leq 2|v|_{\st}^2$ for any $v \in \bR^m$ and $x \in \xi\big(\eB_g\big(p,\sqrt{2}r\big)\big)$;
\item[iv)] $r^{k+\a}\,{\cdot}\,|\!|((\xi^{-1})^*g)_{ij}|\!|_{\cC^{k,\a}(\ol{B_{\st(0,r)}})} \leq 1$ for any $1 \leq i,j \leq m$.
\end{itemize} \label{harmcoord} \end{definition}

In the definition above, $\D_g$ is the Laplace-Beltrami operator of $(M,g)$. Notice that by (i) and (iii) it follows that $B_{\st}(0,r) \subset \xi\big(\eB_g\big(p,\sqrt{2}r\big)\big) \subset B_{\st}(0,2r)$. Moreover, by (ii), (iv) and the classical Schauder interior estimates (see \cite[Thm 6.2 and (4.17)]{GT}), there exists a constant $C=C(m,k,\a,r)>0$ such that $$|\!|f \circ \xi^{-1}|\!|_{\cC^{k+1,\a}(\ol{B_{\st}(0,\frac{r}2)})} \,\leq\, C \!\sup_{\xi^{-1}(B_{\st}(0,r))}\! \big|f\big|$$ for any $\D_g$-harmonic function $f: \xi^{-1}(B_{\st}(0,r)) \subset M \rar \bR$ (see also \cite[Sec 1]{Ka}). \smallskip

\subsection{Local groups of isometries} \label{loctrg} \hfill \par

We collect here some definitions concerning local transformation groups on manifolds. We refer to \cite[Sec 3]{Ped2} for more details.

A {\it local topological group} is a tuple $\fG=(\fG, e, \eD(\fG), \jmath, \nu)$ formed by a Hausdorff topological space $\fG$ with a distinguished element $e \in \fG$ called {\it unit}, an open subset $\eD(\fG) \subset \fG\times\fG$ which contains both $\fG\times\{e\},\{e\}\times \fG$ and two continuous maps $\jmath: \fG \rar \fG$, $\nu: \eD(\fG) \rar \fG$ such that, setting $a_1\cdot a_2 \= \nu(a_1,a_2)$ and $a^{-1}\=\jmath(a)$, the following conditions hold: \begin{itemize}
\item[$\bcdot$] $a \cdot e = e \cdot a = a$,
\item[$\bcdot$] $a_1 \cdot (a \cdot a_2)=(a_1 \cdot a) \cdot a_2$ provided that both side of the equation are well defined,
\item[$\bcdot$] $(a,a^{-1}),(a^{-1},a) \in \eD(\fG)$ and $a \cdot a^{-1} = a^{-1} \cdot a =e$. \end{itemize} Let now $(M,g,p)$ be a pointed $\cC^{k,\a}$-Riemannian manifold, with $k+\a>0$. A {\it local topological group of isometries on $(M,g,p)$} is a tuple $G=(\fG,\eU_{\fG},\W_p,\eW,\Q)$ formed by \begin{itemize}
\item[i)] a local topological group $\fG$ and a neighborhood $\eU_{\fG} \subset \fG$ of the unit $e \in \fG$;
\item[ii)] a neighborhood $\W_p \subset M$ of $p$;
\item[iii)] an open set $\eW \subset \eU_{\fG} \times \W_p$ such that $\eU_{\fG}\times\{p\}$, $\{e\} \times \W_p \subset \eW$ and a continuous application $\Q: \eW \rar \W_p$;
\end{itemize} such that the following hold: \begin{itemize}
\item[$\bcdot$] for any $a \in \eU_{\fG}$, the map $\Q(a) \= \Q(a,\cdot)$ is a local isometry defined in a neighborhood of $p$;
\item[$\bcdot$] for any $(a,b)\in (\eU_{\fG}\times\eU_{\fG}) \cap \eD(\fG)$ and $x \in \W_p$ it holds $\big(\Q(a)\circ\Q(b)\big)(x) =\Q(a\cdot b)(x)$, provided that both side of the equation are well defined;
\item[$\bcdot$] $\Q(e,x)=x$ for any $x \in \W_p$.
\end{itemize} For the sake of shortness, we often identify elements $a \in \eU_{\fG} \subset \fG$ with the corresponding local isometries $f\=\Q(a)$ and we just write ``$f \in G$ ''. The local topological group of isometries $G$ is said to be {\it effective} if $\big\{f \in G : \text{$f$ fixes a neighborhood of $p$}\big\}=\{\Id\}$ and {\it transitive} if the {\it orbit of $G$ through $p$} \begin{multline*} G(p) \= \Big\{\big(f_1 \circ {\dots} \circ f_N \big)(p) : N\geq1 \, , \,\, f_i \in G \, \text{ for any $1\leq i \leq N$} \, , \\ \big(f_{j+1} \circ{\dots}\circ f_N\big)(p) \text{ is well defined for any $1\leq j \leq N-1$} \Big\} \end{multline*} contains a neighborhood of the point $p$.

Finally, if the metric $g$ is smooth, a local topological group of isometries $G=(\fG,\eU_{\fG},\W_p,\eW,\Q)$ acting on $(M,g,p)$ is called {\it local Lie group of isometries} if $\fG$ is a Lie group and the map $\Q$ is of class $\cC^{\infty}$. In this case, the {\it infinitesimal generators} of $G$ are the Killing vector fields defined on $\W_p$ obtained by differentiating the action $\Q$.

\subsection{Locally homogeneous Riemannian spaces} \label{Nomizualg} \hfill \par

A $\cC^{k,\a}$-Riemannian manifold $(M,g)$ is {\it locally homogeneous} if the pseudogroup of local isometries of $(M,g)$ acts transitively on $M$, i.e. if for any $x,y \in M$ there exist two open sets $U_x , U_y \subset M$ and a local isometry $f : U_x \rar U_y$ such that $x \in U_x$, $y \in U_y$ and $f(x) = y$. From now on, we use the nomenclature {\it locally homogeneous Riemannian space} to denote locally homogeneous Riemannian manifolds which are smooth, and hence real analytic (see e.g. \cite[Thm 2.2]{Sp1} or \cite[Lemma 1.1]{BLS}). \smallskip

Given a locally homogeneous space $(M,g)$ and a distinguished point $p \in M$, it is known that there exists a local Lie group of isometries which acts transitively on $(M,g,p)$ according to the definitions in Section \ref{loctrg}. Such local Lie group is constructed as follows. Consider the {\it Killing generators at $p$}, that is the pair $(v,A) \in T_pM \oplus \so(T_pM,g_p)$ such that \beq A \cdot g_p = 0 \,\, , \quad v \,\lrcorner\, \big((\n^g)^{k+1}\Rm(g)_p\big) + A \cdot \big((\n^g)^k\Rm(g)_p\big) = 0 \quad \text{ for any $k \in \bZ_{\geq0}$} \, , \label{killgen} \eeq where $\so(T_pM,g_p)$ acts on the tensor algebra of $T_pM$ as a derivation. The space of Killing generators at $p$ is denoted by $\mathfrak{kill}^g$. It is a Lie algebra with the Lie brackets \beq \big[(v,A),(w,B)\big] \= \big(A(w)-B(v),[A,B]+\Rm(g)_p(v{\wedge}w)\big) \eeq and it is called the {\it Nomizu algebra of $(M,g,p)$}. Since $(M,g)$ is real analytic, by \cite[Thm 2]{No} there exists a neighborhood $\W_p \subset M$ of $p$ such that for any $(v,A) \in \mathfrak{kill}^g$ there exists a Killing vector field $X$ on $\W_p$ with $X_p=v$ and $-(\n^gX)_p=A$. By \cite[Thm XI]{Pa} (see also \cite[Thm A.4]{BL}) there exists a local Lie group of isometries whose infinitesimal generators are exactly the Killing vector fields in $\mathfrak{kill}^g$. \smallskip

This fact admits a converse, namely

\begin{theorem}[\cite{Ped2}, Thm B] Let $(M,g)$ be a locally homogeneous $\cC^1$-Riemannian manifold. If there exist a point $p\in M$ and a locally compact, effective local topological group of isometries which acts transitively on $(M,g,p)$, then $(M,g)$ is real analytic. \label{lochomreg} \end{theorem}

\section{The space of geometric models} \label{spaceGM} \setcounter{equation} 0

\subsection{Geometric models} \hfill \par

We recall that a geometric model is by definition (see the introduction) a $\cC^{\infty}$-smooth locally homogeneous Riemannian distance ball $(\eB, \hat{g})=(\eB_{\hat{g}}(o,\pi), \hat{g})$ of radius $\pi$ satisfying $|\sec(\hat{g})|\leq 1$ and $\inj_o(\eB, \hat{g})=\pi$. From now on, up to pulling back the metric via the Riemannian exponential map, any geometric model will be always assumed to be of the form $(B^m,\hat{g})$, where $B^m \= B_{\st}(0,\pi) \subset \bR^m$ is the $m$-dimensional Euclidean ball of radius $\pi$, and the standard coordinates of $B^m$ will be always assumed to be normal for $\hat{g}$ at $0$. Therefore the geodesics starting from $0 \in B^m$ are exactly the straight lines and $\td_{\hat{g}}(0,x)=|x|_{\st}$ for any $x \in B^m$. Hence, $\eB_{\hat{g}}(0,r)=B_{\st}(0,r)$ for any $0<r \leq \pi$. Moreover, by \cite[Lemma 1.3]{BLS} it follows that \beq \inj_x(B^m,\hat{g})=\pi-|x|_{\st}  \quad \text{ for any $x \in B^m$ }\, . \label{inj} \eeq 

Since $(B^m,\hat{g})$ is real analytic, any local isometry can be extended uniformly, i.e.

\begin{lemma}[\cite{BLS}, Lemma 1.4] For any $x \in B^m$ and for any isometry $f: B_{\st}(0,\e) \rar \eB_{\hat{g}}(x,\e)$ such that $f(0)=x$, there exists an isometry $\tilde{f}: B_{\st}(0,\pi{-}|x|_{\st}) \rar \eB_{\hat{g}}(x,\pi{-}|x|_{\st})$ such that $\tilde{f}|_{B_{\st}(0,\e)}=f$. \label{ext} \end{lemma}

\noindent and, by repeating the same argument, one can also prove the following

\begin{lemma} Let $(B^m,\hat{g}_1)$ and $(B^m,\hat{g}_2)$ be two geometric models. Then, any isometry $$f: (B_{\st}(0,\e),\hat{g}_1) \rar (B_{\st}(0,\e),\hat{g}_2) \quad \text{ with $0 < \e < \pi$}$$ can be uniquely extended to a global isometry $\tilde{f}: (B^m,\hat{g}_1) \rar (B^m,\hat{g}_2)$. \label{uniqueisom} \end{lemma}

\begin{remark} As pointed out in \cite[Sec 1]{BLS}, it is easy to observe that any $m$-dimensional globally homogeneous space $(M,g)$ satisfying $|\sec(g)|\leq1$ is locally isometric to a geometric model $(B^m,\hat{g})$, which is unique up to isometry by Lemma \ref{uniqueisom}. In fact, $(M,g)$ is complete (see \cite[Ch IV, Thm 4.5]{KN1}) and, fixing a point $p \in M$, by the Rauch Comparison Theorem (see e.g. \cite[Sec 6.5]{Jo}), the differential of the Riemannian exponential $\Exp(g)_p: T_pM \rar M$ of $(M,g)$ at $p$ is injective at every point of $B_{g_p}(0_p,\pi) \subset T_pM$. By choosing an orthonormal frame $u: \bR^m \rar T_pM$, one can consider the pulled-back metric $\hat{g}\= (\Exp(g)_p \circ u)^*g$ on $B^m = B_{\st}(0,\pi) \subset \bR^m$. Then, it is easy to check that $(B^m,\hat{g})$ is a geometric model and that it is locally isometric to $(M,g)$. \label{remgmHstar} \end{remark}

Notice that this argument cannot be used to prove that any $m$-dimensional, possibly incomplete locally homogeneous space $(M,g)$ verifying $|\sec(g)|\leq1$ is locally isometric to a geometric model $(B^m,\hat{g})$. Nonetheless, this claim is true and we will prove it in Section \ref{exsec}.
\smallskip

\subsection{Proof of Theorem \ref{MAIN-B}} \hfill \par

Let $(B^m,\hat{g}^{(n)})$ be a sequence of geometric models. The main purpose of this section is to prove the following 

\begin{theorem} For any $0<\a<1$, the sequence $(B^m,\hat{g}^{(n)})$ subconverges to a limit geometric model $(B^m,\hat{g}^{(\infty)})$ in the pointed $\cC^{1,\a}$-topology. \label{thmcomp} \end{theorem}

\noindent which in turn implies Theorem \ref{MAIN-B}. We begin with the following

\begin{prop} The sequence $(B^m,\hat{g}^{(n)})$ subconverges to an incomplete pointed $\cC^{1,\a}$-Riemannian manifold $\big(M^{(\infty)},g^{(\infty)},o\big)$ in the pointed $\cC^{1,\a}$-topology. \end{prop}

\begin{proof} We use local mollifications in the sense of Hochard \cite{Ho}. More concretely, fix a sequence $(\e_k) \subset \bR$ with $0<\e_k<<1$ and $\e_k \rar 0$. By \cite[Lemma 6.2]{Ho} (see also \cite[Lemma 4.3]{ST}) and \eqref{inj} there exists a uniform constant $c=c(m)\geq1$ and, for any fixed $k \in \bN$, there exist open sets $$B_{\st}(0,\pi-2\e_k) \subset M^{(n)}_k \subset B_{\st}(0,\pi-\e_k)$$ and smooth Riemannian metrics $g^{(n)}_k$ on $M^{(n)}_k$ such that: \begin{itemize}
\item[i)] $\big(M^{(n)}_k, g^{(n)}_k\big)$ are complete smooth Riemannian manifolds,
\item[ii)] $g^{(n)}_k = \hat{g}^{(n)}$ on $B_{\st}(0,\pi-3\e_k)$,
\item[iii)] $\sup\big\{\big|\Rm(g^{(n)}_k)_x\big|_{\hat{g}^{(n)}_k} : x \in M^{(n)}_k\big\}<c\,\e_k^{-2}$.
\end{itemize} From (ii) it follows that ${\rm inj}_0\big(M^{(n)}_k,g^{(n)}_k\big) \geq \pi-3\e_k$. Therefore by (i), (iii) and the Cheeger-Gromov precompactness Theorem (see e.g. \cite[Thm 15]{HH}) we can extract a subsequence in such a way that $\big(M^{(n)}_1,g^{(n)}_1,0\big)$ converges in the pointed $\cC^{1,\a}$-topology to a pointed $\cC^{1,\a}$-Riemannian manifold $\big(M^{(\infty)}_1,g^{(\infty)}_1,o_1\big)$ as $n \rar +\infty$. Iterating this construction for any $k \in \bN$ and using a Cantor diagonal procedure, we can extract a subsequence in such a way that $\big(M^{(n)}_k,g^{(n)}_k,0\big)$ converges in the pointed $\cC^{1,\a}$-topology to a pointed $\cC^{1,\a}$-Riemannian manifold $\big(M^{(\infty)}_k,g^{(\infty)}_k,o_k\big)$ as $n \rar +\infty$ for any fixed $k \in \bN$.

In particular, for any $k \in \bN$ we have a sequence of $\cC^{2,\a}$-embeddings $$\psi^{(n)}_k: \eB_{g^{(\infty)}_k}(o_k,\pi-4\e_k) \subset M^{(\infty)}_k \rar M^{(n)}_k$$ such that $\psi_k^{(n)}(o_k)=0$ and $(\psi^{(n)}_k)^*g^{(n)}_k$ converges in the $\cC^{1,\a}$-topology to $g^{(\infty)}_k$ in $\eB_{g^{(\infty)}_k}(o_k,\pi-4\e_k)$ as $n \rar +\infty$. This implies that $\psi_k^{(n)}\big(\eB_{g_k}(o_k,\pi-4\e_k)\big) \subset B_{\st}(0,\pi-3\e_k)$ for $n \in \bN$ sufficiently large and therefore by (ii) it follows that $(\psi^{(n)}_k)^*\hat{g}^{(n)}$ converges in the $\cC^{1,\a}$-topology to $g^{(\infty)}_k$ in $\eB_{g^{(\infty)}_k}(o_k,\pi-4\e_k)$ as $n \rar +\infty$. Hence, for any $k \in \bN$ there exists a pointed isometric embedding $$\f_k: \big(\eB_{g^{(\infty)}_k}(o_k,\pi-4\e_k),g^{(\infty)}_k,o_k\big) \rar \big(\eB_{g^{(\infty)}_{k+1}}(o_{k+1},\pi-4\e_{k+1}),g^{(\infty)}_{k+1},o_{k+1}\big) \,\, .$$ Therefore, we may consider the direct limit (see e.g. \cite[Ch 4, Sec 2.3]{Cho}) $$\big(M^{(\infty)},g^{(\infty)},o\big) \= \lim_{\longrightarrow}\Big\{\big(\eB_{g^{(\infty)}_k}(o_k,\pi-4\e_k),g^{(\infty)}_k,o_k\big),\f_k\Big\}_{k\in \bN} \,\, .$$ The triple $\big(M^{(\infty)},g^{(\infty)},o\big)$ is a pointed $\cC^{1,\a}$-Riemannian manifold. Moreover, by construction we get the statement. \end{proof}

We pass to a subsequence of $(B^m,\hat{g}^{(n)})$ such that there exist an exhaustion $(U^{(n)})$ of $M^{(\infty)}$ by relatively compact open sets centered at $o$ and a sequence of $\cC^{2,\a}$-diffeomorphisms $\phi^{(n)}: U^{(n)} \rar B_{\st}\big(0,\pi{-}{\textstyle\frac1{2^n}}\big)$ such that $\phi^{(n)}(o)=0$ and $g^{(n)}\=\phi^{(n)*}\hat{g}^{(n)}$ converges in the $\cC^{1,\a}$-topology to $g^{(\infty)}$. \smallskip

Let us define the subsets $$\begin{gathered}
\cU^{(n)} \= \bigcup_{p\in U^{(n)}} \{p\}\times B_{g^{(n)}_p}\big(0_p,\pi-{\textstyle\frac1{2^n}}-|\phi^{(n)}(p)|_{\st}\big) \subset TM^{(\infty)} \,\, , \\
\cU^{(\infty)} \= \bigcup_{p\in M^{(\infty)}} \{p\}\times B_{g^{(\infty)}_p}\big(0_p,\pi-\td_{g^{(\infty)}}(o,p)\big) \subset TM^{(\infty)} \,\, .
\end{gathered}$$ and the subsets $$\begin{gathered}
\cV^{(n)} \= \bigcup_{p\in U^{(n)}} \{p\}\times \eB_{g^{(n)}}\big(p,\pi-{\textstyle\frac1{2^n}}-|\phi^{(n)}(p)|_{\st}\big) \subset M^{(\infty)}\times M^{(\infty)} \,\, , \\
\cV^{(\infty)} \= \bigcup_{p\in M^{(\infty)}} \{p\}\times \eB_{g^{(\infty)}}\big(p,\pi-\td_{g^{(\infty)}}(o,p)\big) \subset M^{(\infty)}\times M^{(\infty)} \,\, .
\end{gathered}$$ Notice that the metrics $g^{(n)}$ on $U^{(n)}$ are merely of class $\cC^{1,\a}$. Nonetheless, for any $n \in \bN$ we can consider $$E^{(n)} : \cU^{(n)} \rar U^{(n)} \,\, , \quad E^{(n)}(p,v) \= (\phi^{(n)})^{-1}\Big(\Exp(\hat{g}^{(n)})\big(\phi^{(n)}(p),(d\phi^{(n)})_p(v)\big)\Big) \,\, .$$ By construction, for any $n \in \bN$ the map $$\check{E}^{(n)}: \cU^{(n)} \rar \cV^{(n)} \,\, , \quad \check{E}^{(n)}(p,v)=\big(p,E^{(n)}(p,v)\big)$$ is a $\cC^{2,\a}$-diffeomorphism and we denote the inverse by $\check{L}^{(n)}: \cV^{(n)} \rar \cU^{(n)}$. Notice that the applications $\check{L}^{(n)}$ are necessarily of the form $\check{L}^{(n)}(p,q)=\big(p,L^{(n)}(p,q)\big)$, with $E^{(n)}\big(p,L^{(n)}(p,q)\big)=q$ for any $(p,q) \in \cV^{(n)}$ and $L^{(n)}\big(p,E^{(n)}(p,v)\big)=v$ for any $(p,v) \in \cU^{(n)}$.

\begin{prop} \label{proplimitExp} The sequence $(\check{E}^{(n)})$ subconverges uniformly on compact sets to a $\cC^{0,1}$-homeomorphism \beq \check{E}^{(\infty)}: \cU^{(\infty)} \rar \cV^{(\infty)} \,\, , \quad \check{E}^{(\infty)}(p,v)=\big(p,E^{(\infty)}(p,v)\big) \,\, . \label{limitExp} \eeq \end{prop}

\begin{proof} The proof is based on an Ascoli-Arzel{\aac} diagonal argument. Firstly, we need to compute the differential of $E^{(n)}$ at a point $(p,v) \in \cU^{(n)}$. For this purpose, we pick two vectors $w_1\in T_pM^{(\infty)}$ and $w_2\in T_v(T_pM^{(\infty)})=T_pM^{(\infty)}$, we consider the parallel transports of the push-forwards $(d\phi^{(n)})_p(v),(d\phi^{(n)})_p(w_2) \in T_{\phi^{(n)}(p)}B^m$ along the $\hat{g}^{(n)}$-geodesic $s \mapsto \Exp(\hat{g}^{(n)})(\phi^{(n)}(p),s(d\phi^{(n)})_p(w_1))$ and we pull them back by using $\phi^{(n)}$. We indicate such paths with $v^{(n)}(w_1;s)$ and $w_2^{(n)}(w_1;s)$, respectively. Moreover, consider the $\hat{g}^{(n)}$-Jacobi field along the $\hat{g}^{(n)}$-geodesic $t \mapsto \Exp(\hat{g}^{(n)})(\phi^{(n)}(p),t(d\phi^{(n)})_p(v))$ with initial conditions $(d\phi^{(n)})_p(w_1),(d\phi^{(n)})_p(w_2) \in T_{\phi^{(n)}(p)}B^m$ and pull it back by using $\phi^{(n)}$. We indicate it with $J^{(n)}_{w_1,w_2}(v;t)$. Then, it is straightforward to check that the differential $$\big(dE^{(n)}\big)_{(p,v)} : T_pM^{(\infty)} \oplus T_pM^{(\infty)} \rar T_{E^{(n)}(p,v)}M^{(\infty)}$$ is given by \beq \begin{aligned}
\big(dE^{(n)}\big)_{(p,v)} (w_1,w_2) &= \frac{\p}{\p s} E^{(n)}\!\Big(E^{(n)}(p,sw_1),t\big(v^{(n)}(w_1;s)+s\,w_2^{(n)}(w_1;s)\big)\Big)\bigg|_{\substack{s=0\\t=1}} \\
&=J^{(n)}_{w_1,w_2}(v;1) \,\, . \label{dE(n)} \end{aligned} \eeq We recall that, since $|\sec(g^{(n)})|\leq1$ by assumption, from the Rauch comparison Theorem (see \cite[Thm 6.5.1, Thm 6.5.2]{Jo}) it follows that for any $t \in [0,1]$ \beq \begin{aligned}
\big|J^{(n)}_{w_1,0}(v;t) \big|_{g^{(n)}} &\leq \cosh\!\big(t|v|_{g^{(n)}}\!\big)|w_1|_{g^{(n)}} \,\, , \\
\frac{\sin\!\big(t|v|_{g^{(n)}}\!\big)}{|v|_{g^{(n)}}}|w_2|_{g^{(n)}} \leq \big|J^{(n)}_{0,w_2}(v;t) \big|_{g^{(n)}} &\leq \frac{\sinh\!\big(t|v|_{g^{(n)}}\!\big)}{|v|_{g^{(n)}}}|w_2|_{g^{(n)}} \,\, .
\end{aligned} \label{JacRau} \eeq By \eqref{dE(n)}, the differential $$\big(d\check{E}^{(n)}\big)_{(p,v)} : T_pM^{(\infty)} \oplus T_pM^{(\infty)} \rar T_pM^{(\infty)} \oplus T_{E^{(n)}(p,v)}M^{(\infty)}$$ of the map $\check{E}^{(n)}$ is given by $$\big(d\check{E}^{(n)}\big)_{(p,v)} (w_1,w_2) = \big(w_1,J^{(n)}_{w_1,w_2}(v;1)\big) \,\, .$$ By \eqref{JacRau} it follows that \beq \begin{aligned}
\big|\big(d\check{E}^{(n)}\big)_{(p,v)} (w_1,w_2)\big|_{g^{(n)}} &\leq |w_1|_{g^{(n)}}+\big|J^{(n)}_{w_1,w_2}(v;1) \big|_{g^{(n)}} \\
&\leq |w_1|_{g^{(n)}} + \big|J^{(n)}_{w_1,0}(v;1) \big|_{g^{(n)}} + \big|J^{(n)}_{0,w_2}(v;1) \big|_{g^{(n)}} \\
&\leq \bigg(1+\cosh\!\big(|v|_{g^{(n)}}\!\big)+\frac{\sinh\!\big(|v|_{g^{(n)}}\!\big)}{|v|_{g^{(n)}}}\bigg)\big(|w_1|_{g^{(n)}}+|w_2|_{g^{(n)}}\big)
\end{aligned} \label{estdcE1} \eeq and also \beq \begin{aligned}
\sqrt2 \, \big|\big(d\check{E}^{(n)}\big)_{(p,v)} (w_1,w_2)\big|_{g^{(n)}} &\geq |w_1|_{g^{(n)}}+\tfrac1{12}\big|J^{(n)}_{w_1,w_2}(v;1) \big|_{g^{(n)}} \\
&\geq |w_1|_{g^{(n)}}-\tfrac1{12}\big|J^{(n)}_{w_1,0}(v;1) \big|_{g^{(n)}}+\tfrac1{12}\big|J^{(n)}_{0,w_2}(v;1) \big|_{g^{(n)}} \quad . \\
&\geq \big(1-\tfrac1{12}\cosh\!\big(|v|_{g^{(n)}}\!\big)\big)|w_1|_{g^{(n)}}+\frac{\sin\!\big(t|v|_{g^{(n)}}\!\big)}{12|v|_{g^{(n)}}}|w_2|_{g^{(n)}} 
\end{aligned} \label{estdcE2} \eeq

From \eqref{dE(n)},\eqref{estdcE1} and \eqref{estdcE2}, since $g^{(n)}$ converges in the $\cC^{1,\a}$-topology to $g^{(\infty)}$, for any compact set $K \subset M$ there exists $\bar{\d}=\bar{\d}(K)>0$ such that, for any fixed $0<\d<\bar{\d}$, there exist $\bar{n}=\bar{n}(K,\d) \in \bN$ and $C=C(K,\d), c=c(K,\d)>0$ such that for any $p \in K$ and $n\geq\bar{n}$ $$\emptyset \neq B_{g^{(\infty)}_p}\big(0_p,\pi-\d-\td_{g^{(\infty)}}(o,p)\big) \subset B_{g^{(n)}_p}\big(0_p,\pi-{\textstyle\frac1{2^n}}-|\phi^{(n)}(p)|_{\st}\big) \subset T_pM^{(\infty)}$$ and for any $v \in T_pM^{(\infty)}$ with $|v|_{g^{(\infty)}} \leq \pi-\d-\td_{g^{(\infty)}}(o,p)$ it holds that \beq \begin{gathered}
\td_{g^{(\infty)}}\big(\check{E}^{(n)}(p,v),(o,o)\big)\leq C <\sqrt2\pi \,\, , \\ 
e^{-c}(|w_1|_{g^{(\infty)}}+|w_2|_{g^{(\infty)}}) < \big| \big(d\check{E}^{(n)}\big)_{(p,v)}(w_1,w_2) \big|_{g^{(\infty)}} < e^c(|w_1|_{g^{(\infty)}}+|w_2|_{g^{(\infty)}}) \,\, . \label{E(n)biL} \end{gathered} \eeq Here, we considered the distance $$\td_{g^{(\infty)}}\big((p_1,q_1),(p_2,q_2)\big) \= \sqrt{\td_{g^{(\infty)}}(p_1,p_2)^2+\td_{g^{(\infty)}}(q_1,q_2)^2}$$ on the product $M^{(\infty)} \times M^{(\infty)}$.

Therefore, \eqref{E(n)biL} implies that the sequence of diffeomorphisms $(\check{E}^{(n)})$ is uniformly locally bi-Lipschitz (see e.g. \cite[Lemma 2.10]{KSS}). Let us fix a compact set $K \subset M$, let $\bar{\d}=\bar{\d}(K)$ be as above and consider a sequence $(\d_i) \subset \bR$ such that $0<\d_{i+1}<\d_i<\bar{\d}$ and $\d_i\rar0$. By combining the Ascoli-Arzel{\aac} Theorem with a Cantor diagonal procedure, from \eqref{E(n)biL} it follows that the restriction $\check{E}^{(n)}\big|_{\cU^{(n)}_K}$ subconverges uniformly on compact sets of $\cU^{(\infty)}_K$, where we have set $\cU^{(n)}_K \= \cU^{(n)} \cap TM^{(\infty)}|_K$ and $\cU^{(\infty)}_K \= \cU^{(\infty)} \cap TM^{(\infty)}|_K$. Now, by considering an exhaustion of $M^{(\infty)}$ by compact sets $K_j$ and by applying again a Cantor diagonal procedure, we get that $(\check{E}^{(n)})$ subconverges uniformly on compact sets to a $\cC^{0,1}$-homeomorphism $$\check{E}^{(\infty)}: \cU^{(\infty)} \rar \cV^{(\infty)} \,\, , \quad \check{E}^{(\infty)}(p,v) = \big(p,E^{(\infty)}(p,v)\big)$$ and this completes the proof. \end{proof}

We pass to a subsequence of $(B^m,\hat{g}^{(n)})$ such that the maps $\check{E}^{(n)}$ converge uniformly on compact sets to a $\cC^{0,1}$-homeomorphism $\check{E}^{(\infty)}$, which is necessarily of the form \eqref{limitExp}. We denote its inverse by $$\check{L}^{(\infty)}: \cV^{(\infty)} \rar \cU^{(\infty)} \,\, , \quad \check{L}^{(\infty)}(p,q) \= \big(p,L^{(\infty)}(p,q)\big) \,\, .$$ By passing to a subsequence, we can assume that $(\check{L}^{(n)})$ converges uniformly on compact sets to $\check{L}^{(\infty)}$. For the sake of shortness, we set $$E^{(n)}_p \= E^{(n)}(p, \cdot) \,\, , \quad E^{(\infty)}_p \= E^{(\infty)}(p, \cdot) \,\, , \quad L^{(n)}_p \= L^{(n)}(p, \cdot) \,\, , \quad L^{(\infty)}_p \= L^{(\infty)}(p, \cdot) \,\, .$$ Clearly it holds that $L^{(n)}_p = \big(E^{(n)}_p\big)^{-1}$ and $L^{(\infty)}_p = \big(E^{(\infty)}_p\big)^{-1}$. \smallskip

In the next theorem we will construct explicitly a local topological group of isometries acting on the limit manifold $(M^{(\infty)},g^{(\infty)},o)$. This will allow us to apply Theorem \ref{lochomreg} afterwards. Firstly, we denote by ${\rm O}_{g^{(\infty)}}\big(M^{(\infty)}\big) \rar M^{(\infty)}$ the orthonormal frame bundle of $(M^{(\infty)},g^{(\infty)})$. Secondly, letting $u_{\st}$ be the standard orthonormal frame of $\big(T_0B^m,\hat{g}^{(n)}_0\big)=\big(\bR^m,\la \cdot, \cdot \ra_{\st}\big)$, we can assume up to pass to a subsequence that $u^{(n)}_{\st} \= ((d\phi^{(n)})_o)^{-1}(u_{\st})$ converges to an orthonormal frame $u^{(\infty)}_{\st}$ for $\big(T_oM^{(\infty)},g^{(\infty)}_o\big)$. This yields an identification $${\rm O}_{g^{(\infty)}}\big(M^{(\infty)}\big) = \Big\{(p,a): p \in M^{(\infty)} \, , \,\, a: \big(T_oM^{(\infty)},g^{(\infty)}_o\big) \rar \big(T_pM^{(\infty)},g^{(\infty)}_p\big) \text{ linear isometry } \Big\}$$

\begin{theorem} There exists a locally compact and effective local topological group of isometries acting transitively on the pointed $\cC^{1,\a}$-Riemannian manifold $(M^{(\infty)},g^{(\infty)},o)$. \label{limitG} \end{theorem}

\begin{proof} Since the elements of the sequence $(B^m,\hat{g}^{(n)})$ are locally homogeneous and smooth, we can pick for any $n\in \bN$ an effective local Lie group of isometries $G^{(n)}$ acting transitively on $(B^m,\hat{g}^{(n)},0)$. By Lemma \ref{ext}, any local isometry $f^{(n)} \in G^{(n)}$ admits a unique analytic extension $$f^{(n)}: B_{\st}\big(0,\pi{-}|f^{(n)}(0)|_{\st}\big) \rar \eB_{\hat{g}^{(n)}}\big(f^{(n)}(0),\pi{-}|f^{(n)}(0)|_{\st}\big) \,\, .$$ For any $n \in \bN$ and for any $f^{(n)} \in G^{(n)}$ such that $|f^{(n)}(0)|_{\st}<\frac{\pi}2$, we define $$\tilde{f}^{(n)}: \Dom(\tilde{f}^{(n)}) \rar \Im(\tilde{f}^{(n)}) \,\, , \quad \tilde{f}^{(n)} \= \phi^{(n)-1} \circ f^{(n)} \circ \big(\phi^{(n)}|_{\Dom(\tilde{f}^{(n)})}\big) \,\, ,$$ where $\Dom(\tilde{f}^{(n)}), \Im(\tilde{f}^{(n)}) \subset U^{(n)}$ are the open subsets given by $$\begin{gathered}
\Dom(\tilde{f}^{(n)}) \= \big(f^{(n)}\circ\phi^{(n)}\big)^{-1}\Big(B_{\st}\big(0,\pi-{\textstyle\frac1{2^n}}\big)\cap \eB_{\hat{g}^{(n)}}\big(f^{(n)}(0),\pi{-}|f^{(n)}(0)|_{\st}\big) \Big) \,\, , \\
\Im(\tilde{f}^{(n)}) \= \big(\phi^{(n)}\big)^{-1}\Big(B_{\st}\big(0,\pi-{\textstyle\frac1{2^n}}\big)\cap \eB_{\hat{g}^{(n)}}\big(f^{(n)}(0),\pi{-}|f^{(n)}(0)|_{\st}\big) \Big) \,\, . \end{gathered}$$ Hence $\tilde{G}^{(n)} \= \big\{\tilde{f}^{(n)}: f^{(n)} \in G^{(n)} , \, |f^{(n)}(0)|_{\st}<\frac{\pi}2 \big\}$ is a transitive, effective local Lie group of isometries on $(M^{(\infty)},g^{(n)},o)$. We are going to specify the local group structure of $\tilde{G}^{(n)}$ below.

The multiplication \beq \tilde{\nu}^{(n)}:\eD(\tilde{G}^{(n)}) \rar \tilde{G}^{(n)} \label{nu(n)} \eeq is defined in the following way. First, we set $$\eD(\tilde{G}^{(n)}) \= \big\{\big(\tilde{f}^{(n)}_1,\tilde{f}^{(n)}_2\big) \in \tilde{G}^{(n)}{\times}\tilde{G}^{(n)} : \big|f^{(n)}_1(f^{(n)}_2(0))\big|_{\st}<{\textstyle\frac{\pi}2}\big\} \,\, .$$ Then, given $\big(\tilde{f}^{(n)}_1,\tilde{f}^{(n)}_2\big) \in \eD(\tilde{G}^{(n)})$, we consider a neighborhood $V$ of $0 \in B^m$ such that $$V \subset B_{\st}\big(0,\pi{-}|f_2^{(n)}(0)|_{\st}\big) \,\, , \quad f_2(V) \subset B_{\st}\big(0,\pi{-}|f_1^{(n)}(0)|_{\st}\big)$$ and we set $f^{(n)}_3 \= f^{(n)}_1 \circ (f^{(n)}_2|_V)$. Then, we consider the analytic extension of $f^{(n)}_3$ as above and we set $\tilde{\nu}^{(n)}\big(\tilde{f}^{(n)}_1,\tilde{f}^{(n)}_2\big) \= \tilde{f}^{(n)}_3$. In the same fashion, the inversion map \beq \tilde{\jmath}^{(n)}: \tilde{G}^{(n)} \rar \tilde{G}^{(n)} \label{jmath(n)} \eeq is defined by choosing, for any given $\tilde{f}^{(n)} \in \tilde{G}^{(n)}$, a neighborhood $V$ of $0 \in B^m$ such that $$V \subset B_{\st}\big(0,\pi{-}|f^{(n)}(0)|_{\st}\big) \,\, , \quad 0 \in f^{(n)}(V)$$ and defining $f'{}^{(n)} \= (f^{(n)}|_{V})^{-1}$. Then, we consider the analytic extension of $f'{}^{(n)}$ and we set $\tilde{\jmath}^{(n)}(\tilde{f}^{(n)})\=\tilde{f'}{}^{(n)}$.

We stress the fact that any $\tilde{f}^{(n)} \in \tilde{G}^{(n)}$ verifies \beq \tilde{f}^{(n)}(x) = \big(E^{(n)}_{\tilde{f}^{(n)}(p)} \circ \big(d\tilde{f}^{(n)}|_p\big) \circ L^{(n)}_p\big)(x) \,\, , \quad x \in \Dom(\tilde{f}^{(n)}) \cap \eB_{g^{(n)}}\big(p,\pi-\tfrac1{2^n}-|\phi^{(n)}(p)|\big) \label{f-exp(n)} \eeq for any $p \in \eB_{g^{(\infty)}}(o,\frac{\pi}2)$ in $M^{(\infty)}$.

Let us consider now a dense and countable subset $S \subset \eB_{g^{(\infty)}}(o,\frac{\pi}2)$. By combining the Ascoli-Arzel{\aac} Theorem with a Cantor diagonal procedure (see Step 1 of the proof of \cite[Thm 6.6]{Heb}), by passing to a subsequence, the following claim holds true: for any $p \in S$, there exists a sequence $\tilde{f}^{(n)} \in \tilde{G}^{(n)}$ which converges in the $\cC^1$-topology to a $g^{(\infty)}$-isometry $\tilde{f}^{(\infty)}: \eB_{g^{(\infty)}}(o,\pi{-}\td_{g^{(\infty)}}(o,p)) \rar \eB_{g^{(\infty)}}(p,\pi{-}\td_{g^{(\infty)}}(o,p))$, with $\tilde{f}^{(\infty)}(o)=p$. For the sake of brevity, we just write $\tilde{f}^{(\infty)}=\lim_{\cC^1}\tilde{f}^{(n)}$ and \beq \Dom(\tilde{f}^{(\infty)}) \= \eB_{g^{(\infty)}}\big(o,\pi{-}\td_{g^{(\infty)}}(o,\tilde{f}^{(\infty)}(o))\big) \,\, . \label{dom} \eeq Let us set $$\tilde{G}^{(\infty)}\=\big\{\tilde{f}^{(\infty)}: \text{ there exists a sequence $\tilde{f}^{(n)} \in \tilde{G}^{(n)}$ s.t. $\tilde{f}^{(\infty)}={\textstyle\lim_{\cC^1}}\tilde{f}^{(n)}$ } \big\} \,\, .$$ Then, $\tilde{G}^{(\infty)}$ can be endowed with a structure of local group in the following way.

Firstly, we define the subset $\eD(\tilde{G}^{(\infty)}) \subset \tilde{G}^{(\infty)} \times \tilde{G}^{(\infty)}$ of those pairs $(\tilde{f}^{(\infty)}_1,\tilde{f}^{(\infty)}_2)$ for which there exist $\tilde{f}^{(n)}_1,\tilde{f}^{(n)}_2 \in \tilde{G}^{(n)}$ such that $\big(\tilde{f}^{(n)}_1,\tilde{f}^{(n)}_2\big) \in \eD(\tilde{G}^{(n)})$ for any $n\in \bN$ large enough and $\tilde{f}^{(\infty)}_i=\lim_{\cC^1}\tilde{f}^{(n)}_i$. Notice that the definition of $\eD(\tilde{G}^{(\infty)})$ does not depend on the choice of the sequences $(\tilde{f}^{(n)}_i)$, $i=1,2$. Then, we set \beq \tilde{\nu}^{(\infty)}: \eD(\tilde{G}^{(\infty)}) \rar \tilde{G}^{(\infty)} \,\, , \quad \tilde{\nu}^{(\infty)}(\tilde{f}^{(\infty)}_1,\tilde{f}^{(\infty)}_2) \= {\textstyle\lim_{\cC^1}}\tilde{\nu}^{(n)}(\tilde{f}^{(n)}_1,\tilde{f}^{(n)}_2) \,\, , \label{tildenu}\eeq where $\tilde{\nu}^{(n)}$ was defined in \eqref{nu(n)}. Analogously, we set \beq \tilde{\jmath}^{(\infty)}: \tilde{G}^{(\infty)} \rar \tilde{G}^{(\infty)} \,\, , \quad \tilde{\jmath}^{(\infty)}(\tilde{f}^{(\infty)}) \= {\textstyle\lim_{\cC^1}}\tilde{\jmath}^{(n)}(\tilde{f}^{(n)}) \,\, , \label{tildejmath} \eeq where $\tilde{\jmath}^{(n)}$ was defined in \eqref{jmath(n)}. One can directly check that both \eqref{tildenu} and \eqref{tildejmath} are well defined.

From Proposition \ref{proplimitExp} and \eqref{f-exp(n)} it follows that any $\tilde{f}^{(\infty)} \in \tilde{G}^{(\infty)}$ verifies \beq \tilde{f}^{(\infty)}(x) = \big(E^{(\infty)}_{\tilde{f}(p)} \circ \big(d\tilde{f}^{(\infty)}|_p\big) \circ L^{(\infty)}_p\big)(x) \,\, , \quad x \in \Dom(\tilde{f}^{(n)}) \cap \eB_{g^{(\infty)}}\big(p,\pi-\td_{g^{(\infty)}}(o,p)\big)  \label{f-exp} \eeq for any $p \in \eB_{g^{(\infty)}}(o,\frac{\pi}2)$ in $M^{(\infty)}$.

Let us consider now the map \beq \s: \tilde{G}^{(\infty)} \rar {\rm O}_{g^{(\infty)}}\big(M^{(\infty)}\big) \,\, , \quad \s(\tilde{f}^{(\infty)}) \= \big(\tilde{f}^{(\infty)}(o),d\tilde{f}^{(\infty)}|_o\big) \,\, .\label{sigma} \eeq From \eqref{f-exp} it comes directly that $\s$ is injective. Then, we denote by $G^{(\infty)}$ the topological closure of $\s(\tilde{G}^{(\infty)})$ inside ${\rm O}_{g^{(\infty)}}\big(M^{(\infty)}\big)$. One can define a structure of local topological group of $G^{(\infty)}$ by defining an open set $\eD(G^{(\infty)}) \subset G^{(\infty)}\times G^{(\infty)}$ and two maps $\nu^{(\infty)}: \eD(G^{(\infty)}) \rar G^{(\infty)}$, $\jmath^{(\infty)}: G^{(\infty)} \rar G^{(\infty)}$ by extending \eqref{tildenu}, \eqref{tildejmath} in the following way.

First, let $\eD(G^{(\infty)})$ be the subset of those pairs $((p_1,a_1),(p_2,a_2)) \in G^{(\infty)} \times G^{(\infty)}$ for which there exist $(\tilde{f}^{(\infty)}_{1,k}),(\tilde{f}^{(\infty)}_{2,k}) \subset \tilde{G}^{(\infty)}$ such that $\big(\tilde{f}^{(\infty)}_{1,k},\tilde{f}^{(\infty)}_{2,k}\big) \in \eD(\tilde{G}^{(\infty)})$ for any $k\in \bN$ large enough and $(p_i,a_i)=\lim_{k\rar+\infty}\s(\tilde{f}^{(\infty)}_{i,k})$. Then, we define \beq \nu^{(\infty)}: \eD(G^{(\infty)}) \rar G^{(\infty)} \,\, , \quad \nu^{(\infty)}((p_1,a_1),(p_2,a_2)) \= \lim_{k\rar+\infty}\s\big(\tilde{\nu}^{(\infty)}(\tilde{f}^{(\infty)}_{1,k},\tilde{f}^{(\infty)}_{2,k})\big) \,\, . \label{nu} \eeq Notice that, if we set $p_{i,k} \=\tilde{f}^{(\infty)}_{i,k}(o)$ and $a_{i,k} \= d\tilde{f}^{(\infty)}_{i,k}|_o$ we get \begin{align*}
\tilde{\nu}^{(\infty)}(\tilde{f^{(\infty)}}_{1,k},\tilde{f}^{(\infty)}_{2,k})(o) &= \big(E^{(\infty)}_{p_{1,k}} \circ a_{1,k} \circ L^{(\infty)}_o \circ E^{(\infty)}_{p_{2,k}} \circ a_{2,k}\big)(0_o) \,\, , \\
d\big(\tilde{\nu}^{(\infty)}(\tilde{f}^{(\infty)}_{1,k},\tilde{f}^{(\infty)}_{2,k})\big)|_o &= L^{(\infty)}_{(E^{(\infty)}_{p_{1,k}} \circ a_{1,k} \circ L^{(\infty)}_o \circ E^{(\infty)}_{p_{2,k}} \circ a_{2,k})(0_o)} \circ E^{(\infty)}_{p_{1,k}} \circ a_{1,k} \circ L^{(\infty)}_o{} \circ E^{(\infty)}_{p_{2,k}} \circ a_{2,k}
\end{align*} and hence the limit \eqref{nu} exists, it does not depend on the choice of $(\tilde{f}^{(\infty)}_{1,k},\tilde{f}^{(\infty)}_{2,k})$ and the map $\nu^{(\infty)}$ is continuous. Analogously, we define \beq \jmath^{(\infty)}: G^{(\infty)} \rar G^{(\infty)} \,\, , \quad \jmath^{(\infty)}(p,a) \= \lim_{k\rar+\infty}\s\big(\tilde{\jmath}^{(\infty)}(\tilde{f}^{(\infty)}_k)\big) \,\, , \label{jmath} \eeq where $(\tilde{f}^{(\infty)}_k) \subset \tilde{G}^{(\infty)}$ is an arbitrary sequence such that $(p,a)=\lim_{k\rar+\infty}\sigma(\tilde{f}^{(\infty)}_k)$. Then, setting $p_k \= \tilde{f}^{(\infty)}_k(o)$ and $a_k \= d\tilde{f}^{(\infty)}|_o$ we get \begin{align*}
\tilde{\jmath}(\tilde{f}^{(\infty)}_k)(o) &= \big(E^{(\infty)}_o \circ a_k^{-1} \circ L^{(\infty)}_{p_k}\big)(o) \,\, , \\
d\big(\tilde{\jmath}^{(\infty)}(\tilde{f}^{(\infty)}_k)\big)|_o &=L^{(\infty)}_{(E^{(\infty)}_o \circ a_k^{-1} \circ L^{(\infty)}_{p_k})(o)} \circ E^{(\infty)}_o \circ a_k^{-1} \circ L^{(\infty)}_{p_k} \circ E^{(\infty)}_o
\end{align*} and hence the limit \eqref{jmath} exists, it does not depend on the choice of $(\tilde{f}^{(\infty)}_k)$ and the map $\jmath^{(\infty)}$ is continuous. Finally, we define the open set \beq \eW^{(\infty)} \= \bigcup_{(p,a) \in G^{(\infty)}} \{(p,a)\} \times \eB_{g^{(\infty)}}(o,\pi-\td_{g^{(\infty)}}(o,p)) \subset G^{(\infty)} \times M^{(\infty)} \label{W}\eeq and the map \beq \Q^{(\infty)}: \eW^{(\infty)} \rar M^{(\infty)} \,\, , \quad \Q^{(\infty)}((p,a),x) \= \lim_{k\rar+\infty} \tilde{f}^{(\infty)}_k(x) \,\, , \label{Q} \eeq where $(\tilde{f}^{(\infty)}_k) \subset \tilde{G}^{(\infty)}$ is an arbitrary sequence such that $(p,a)=\lim_{k\rar+\infty}\sigma(\tilde{f}^{(\infty)}_k)$. Again, the limit \eqref{Q} exists, it does not depend on $(\tilde{f}^{(\infty)}_k)$ and the map $\Q^{(\infty)}$ is continuous. From now on, we identify any pair $(p,a) \in G^{(\infty)}$ with the corresponding map $f^{(\infty)} \= \Q^{(\infty)}((p,a),\cdot)$ and we just write ``$f^{(\infty)} \in G^{(\infty)}$ ''. Notice that the map $f^{(\infty)} \in G^{(\infty)}$ corresponding to a pair $(p,a)$ is given by \beq f^{(\infty)} = E^{(\infty)}_p \circ a \circ L^{(\infty)}_o\big|_{\Dom(f^{(\infty)})} \,\, , \quad \Dom(f^{(\infty)}) \= \eB_{g^{(\infty)}}\big(o,\pi{-}\td_{g^{(\infty)}}(o,p)\big) \,\, . \label{finG} \eeq

From \eqref{Q} it follows that each map $f^{(\infty)} \in G^{(\infty)}$ is an isometry, and hence $G^{(\infty)}$ is a local topological group of isometries on $(M^{(\infty)},g^{(\infty)},o)$. Since it is a closed subset of ${\rm O}_{g^{(\infty)}}\big(M^{(\infty)}\big)$, it is locally compact. Moreover, by \eqref{finG} it is effective. Finally, since by construction the orbit $G^{(\infty)}(o)$ contains the whole ball $\eB_{g^{(\infty)}}(o,\frac{\pi}2)$, we conclude that $G^{(\infty)}$ is transitive. \end{proof}

Finally, we are ready to prove Theorem \ref{thmcomp}.

\begin{proof}[Proof of Theorem \ref{thmcomp}] Since the elements of the sequence $(B^m,\hat{g}^{(n)})$ are locally homogeneous, the limit space $\big(M^{(\infty)},g^{(\infty)}\big)$ is locally homogeneous (see \cite[proof of Thm 1.6]{BLS}). Hence, by Theorem \ref{lochomreg} and Theorem \ref{limitG} it follows that $\big(M^{(\infty)},g^{(\infty)}\big)$ is a locally homogeneous $\cC^{\w}$-Riemannian manifold. Since convergence in pointed $\cC^{1,\a}$-topology implies convergence in pointed Gromov-Hausdorff topology, it holds that $\sec(g^{(\infty)})\geq-1$. Moreover, by \cite[Thm. 6.4.8]{Pet} we get a uniform positive lower bound on the convexity radius ${\rm conv}_{0}(B^m,\hat{g}^{(n)})$ along the sequence and hence by \cite[Thm. 5.1]{BS} we get $\sec(g^{(\infty)})\leq1$. Furthermore by \cite[Lemma 1.5]{Sak} it necessarily holds that $E^{(\infty)}=\Exp(g^{(\infty)})$ (see also the proof of \cite[Thm 4.4]{Pe}). Therefore, fixing a $g^{(\infty)}$-orthonormal frame $u: \bR^m \rar T_oM$, one can consider the pulled-back metric $\hat{g}^{(\infty)}\= (\Exp(g^{(\infty)})_o \circ u)^*g^{(\infty)}$ on $B^m = B_{\st}(0,\pi) \subset \bR^m$. It is easy to realize that $(B^m,\hat{g}^{(\infty)})$ is a geometric model and, since it is isometric to $(M^{(\infty)},g^{(\infty)})$, we conclude that $(B^m,\hat{g}^{(n)})$ converges in the pointed $\cC^{1,\a}$-topology to $(B^m,\hat{g}^{(\infty)})$. \end{proof}

From Theorem \ref{thmcomp} we also get the following

\begin{corollary} Let $(B^m,\hat{g}^{(n)})$ be a sequence of geometric models and assume that there exist an integer $k\geq0$ and a constant $C>0$ such that $$\sum_{i=0}^k\big|(\n^{\hat{g}^{(n)}})^i\Rm(\hat{g}^{(n)})\big|_{\hat{g}^{(n)}} \leq C \,\, . $$ Then, for any $0<\a<1$, the sequence $(B^m,\hat{g}^{(n)})$ subconverges to a limit geometric model $(B^m,\hat{g}^{(\infty)})$ in the pointed $\cC^{k+1,\a}$-topology. \label{corcomp} \end{corollary}

\begin{proof} By means of Theorem \ref{thmcomp}, we can assume that $(B^m,\hat{g}^{(n)})$ converges to a geometric model $(B^m,\hat{g}^{(\infty)})$ in the pointed $\cC^{1,\a}$-topology as $n \rar +\infty$. Fix $0 <\d < \pi$ and set $\W_{\d} \= B_{\st}(0,\pi-\d) \subset B^m$. Then (see e.g. \cite[Thm 6]{HH}) it follows that $${\rm har}^{k+1,\a}_x\big(B^m,\hat{g}^{(n)}\big)\geq r_{\zero} >0 \quad \text{ for any $x \in \W_{\d}$ }$$ (see Definition \ref{harmcoord}). Following \cite[proof of Thm A]{Ka} and \cite[Lemma 2.1]{And}, it is possible to construct, up to passing to a subsequence, smooth embeddings $\psi^{(n)}: \W_{\d} \rar \bR^N$, for some $N>>m$, such that $\psi^{(n)}(0)=0$ and $\W^{(n)}_{\d} \= \psi^{(n)}(\W_{\d}) \subset \bR^N$ are locally represented as graphs of smooth functions over $B_{\st}\big(0,\tfrac{r_{\zero}}2\big) \subset \bR^m$ uniformly bounded in $\cC^{k+2,\a}\big(\overline{B_{\st}\big(0,\tfrac{r_{\zero}}2\big)}\big)$. Then, by passing to a subsequence, $\W^{(n)}_{\d}$ converges in the $\cC^{k+2,\a}$-topology as submanifolds of $\bR^N$ to an embedded $\cC^{k+2,\a}$-submanifold $\W^{(\infty)}_{\d}$ of $\bR^N$. Set $g^{(n)} \=((\psi^{(n)})^{-1})^*\hat{g}^{(n)}$. By passing to a further subsequence, the projection along the normals of $\W^{(\infty)}_{\d}$ induces a sequences of $\cC^{k+2,\a}$-diffeomorphisms ${\rm pr}^{(n)}:\W^{(n)}_{\d} \rar \W^{(\infty)}_{\d}$ such that the metrics $(({\rm pr}^{(n)})^{-1})^*g^{(n)}$ converge in the $\cC^{k+1,\a}$-topology to a $\cC^{k+1,\a}$-Riemannian metric $g^{(\infty)}$ on $\W^{(\infty)}_{\d}$.

Since we have assumed that $(B^m,\hat{g}^{(n)})$ converges to $(B^m,\hat{g}^{(\infty)})$ in the pointed $\cC^{1,\a}$-topology, it follows that there exists an isometric embedding $\f: (\W^{(\infty)}_{\d},g^{(\infty)}) \rar (B^m,\hat{g}^{(\infty)})$ with $\f(0)=0$. Moreover, given any compact set $K \subset B^m$, one can take $\d>0$ small enough in such a way that $K \subset \f\big(\W^{(\infty)}_{\d}\big)$. This completes the proof. \end{proof}

\section{Existence of geometric models} \label{exsec} \setcounter{equation} 0

Let $(M^m,g)$ be a locally homogeneous space. By Section \ref{Nomizualg}, we can fix a distinguished point $p \in M$ and consider the Nomizu algebra $\mathfrak{kill}^g=\{ \text{ Killing generators $(v,A)$ at $p$ } \}$ of $(M,g,p)$. Consider the Euclidean scalar product on $\mathfrak{kill}^g$ given by $$\la\!\la(v,A),(w,B)\ra\!\ra_g\=g_p(v,w)-\Tr(AB) \,\, ,$$ set $\mathfrak{kill}^g_0\=\{(0,A) \in \mathfrak{kill}^g\}$ and let $\gm$ be the $\la\!\la\,,\ra\!\ra_g$-orthogonal complement of $\mathfrak{kill}^g_0$ in $\mathfrak{kill}^g$. Since $(M,g)$ is locally homogeneous, the projection $(v,A) \mapsto v$ gives rise to a linear isomorphism $\gm \simeq T_pM$, which allows us to define a scalar product $\la\,,\ra_g$ on $\gm$ induced by the metric tensor $g$. An {\it adapted frame for $\mathfrak{kill}^g$} is a basis $u=(e_1,{\dots},e_{q+m}): \bR^{q+m} \rar \mathfrak{kill}^g$ such that $\mathfrak{kill}^g_0=\vspan(e_1,{\dots},e_q)$, $\gm=\vspan(e_{q+1},{\dots},e_{q+m})$ and $\la e_{q+i},e_{q+j}\ra_g=\d_{ij}$, with $q \= \dim \mathfrak{kill}^g_0$. Given an adapted frame $u$ for $\mathfrak{kill}^g$, we will denote the induced orthonormal frame for $(M,g)$ at $p$ by $u|_{\bR^m}: \bR^m \rar \gm = T_pM$.

\begin{definition} A sequence of $m$-dimensional locally homogeneous spaces $(M^{(n)},g^{(n)})$ {\it converges algebraically} to a locally homogeneous space $(M^{(\infty)},g^{(\infty)})$ if $\dim M^{(\infty)}=m$, $\dim\big(\mathfrak{kill}^{g^{(n)}}_0\big)\equiv\dim\big(\mathfrak{kill}^{g^{(\infty)}}_0\big)$ and there exist distinguished points $p^{(n)} \in M^{(n)}$, $p^{(\infty)} \in M^{(\infty)}$ together with adapted frames $u^{(n)}$, $u^{(\infty)}$ for the Nomizu algebras $\mathfrak{kill}^{g^{(n)}}$, $\mathfrak{kill}^{g^{(\infty)}}$ of $(M^{(n)},g^{(n)},p^{(n)})$, $(M^{(\infty)},g^{(\infty)},p^{(\infty)})$, respectively, such that the structure constants of the Lie algebra $\mathfrak{kill}^{g^{(n)}}$ with respect to $u^{(n)}$ converge to those of $\mathfrak{kill}^{g^{(\infty)}}$ with respect to $u^{(\infty)}$. \label{algconv} \end{definition}

Notice that the notion of algebraic convergence does not depend on the choices either of the base points $p^{(n)}$ and $p^{(\infty)}$, or of the adapted frames $u^{(n)}$ and $u^{(\infty)}$. Moreover, the following facts hold true.

\begin{fact} If $(M^{(n)},g^{(n)})$ converges algebraically to $(M^{(\infty)},g^{(\infty)})$ and $p^{(n)}$, $p^{(\infty)}$, $u^{(n)}$, $u^{(\infty)}$ are taken as in Definition \ref{algconv}, then for any integer $k \geq 0$ $$\big(u^{(n)}|_{\bR^m}\big)^*\Big(\big(\n^{g^{(n)}}\big)^k\Rm(g^{(n)}) \big|_{p^{(n)}}\Big) \,\,\longrightarrow\,\, \big(u^{(\infty)}|_{\bR^m}\big)^*\Big(\big(\n^{g^{(\infty)}}\big)^k\Rm(g^{(\infty)}) \big|_{p^{(\infty)}}\Big) \quad \text{ as $n \rar +\infty$ } $$ in the standard Euclidean topology of $\bigotimes^k(\bR^m)^*\otimes\L^2(\bR^m)^*\otimes\so(m)$. \label{fact1} \end{fact}

\begin{fact} For any locally homogeneous space $(M,g)$, there exists a sequence $(M^{(n)},g^{(n)})$ of globally homogeneous spaces which converges algebraically to $(M,g)$. \label{fact2} \end{fact}

Fact \ref{fact1} and Fact \ref{fact2} are known in the literature (see e.g. \cite{Ts}). However, for the convenience of the reader, we will provide a proof of these two facts on the forthcoming paper \cite{Ped3}.

\begin{proof}[Proof of Theorem \ref{MAIN-A}] Let $(M,g)$ be a strictly locally homogeneous space with $|\sec(g)|\leq 1$. By Fact \ref{fact2}, there exists a sequence $(M^{(n)},g^{(n)})$ of globally homogeneous spaces which converges algebraically to $(M,g)$. By Fact \ref{fact1} it follows that there exists a sequence $(\e^{(n)}) \subset (0,\pi)$ such that $$\big|\sec(g^{(n)})\big|\leq \frac1{\big(1-\frac{\e^{(n)}}{\pi}\big)^2} \,\, , \quad \e^{(n)}\to0 \,\, .$$ By repeating the same argument as in Remark \ref{remgmHstar}, we can pull back the metric $g^{(n)}$ to the tangent ball $B_{{g^{(n)}}_{p^{(n)}}}(0_{p^{(n)}},\pi-\e^{(n)}) \subset T_{p^{(n)}}M^{(n)}$ of radius $\pi-\e^{(n)}$ at some point $p^{(n)} \in M^{(n)}$ via the Riemannian exponential $\Exp(g^{(n)})_{p^{(n)}}:T_{p^{(n)}}M^{(n)} \rar M^{(n)}$. By Fact \ref{fact1} and Corollary \ref{corcomp} we can pass to a subsequence in such a way that $\big(B_{{g^{(n)}}_{p^{(n)}}}(0_{p^{(n)}},\pi-\e^{(n)}), \Exp(g^{(n)})^*g^{(n)}\big)$ converges in the pointed $\cC^{\infty}$-topology to a geometric model $(B^m,\hat{g}^{(\infty)})$. Finally, since the curvature tensor and all its covariant derivatives at a point $p$ determines a complete set of invariants for real analytic Riemannian manifolds up to local isometry around $p$ (see e.g. \cite[Cor E.III.8]{BGM}), it follows that $(B^m,\hat{g}^{(\infty)})$ is locally isometric to $(M,g)$. \end{proof}

\begin{proof}[Proof of Proposition \ref{MAIN-C}] Let $(\eB^{(n)},\hat{g}^{(n)})$ be a sequence of geometric models and assume that it converges algebraically to a locally homogeneous limit space $(M^{(\infty)},g^{(\infty)})$. By Fact \ref{fact1}, Theorem \ref{MAIN-A} and Corollary \ref{corcomp}, by arguing as in the proof of Theorem \ref{MAIN-A} it follows that $\big|\sec(g^{(\infty)})\big|\leq1$ and that there exists a subsequence of $(\eB^{(n)},\hat{g}^{(n)})$ converging to the geometric model $(\eB^{(\infty)},\hat{g}^{(\infty)})$ of $(M^{(\infty)},g^{(\infty)})$ in the pointed $\cC^{\infty}$-topology. Moreover, again by Fact \ref{fact1} and \cite[Cor E.III.8]{BGM}, any convergent subsequence of $(\eB^{(n)},\hat{g}^{(n)})$ in the pointed $\cC^{\infty}$-topology necessarily converges to $(\eB^{(\infty)},\hat{g}^{(\infty)})$. Therefore, we conclude that the whole sequence $(\eB^{(n)},\hat{g}^{(n)})$ converges to $(\eB^{(\infty)},\hat{g}^{(\infty)})$ in the pointed $\cC^{\infty}$-topology. \end{proof}

\bigskip\bigskip
\font\smallsmc = cmcsc8
\font\smalltt = cmtt8
\font\smallit = cmti8
\hbox{\parindent=0pt\parskip=0pt
\vbox{\baselineskip 9.5 pt \hsize=5truein
\obeylines
{\smallsmc
Dipartimento di Matematica e Informatica ``Ulisse Dini'', Universit$\scalefont{0.55}{\text{\Aac}}$ di Firenze
Viale Morgagni 67/A, 50134 Firenze, ITALY}
\smallskip
{\smallit E-mail adress}\/: {\smalltt francesco.pediconi@unifi.it
}
}
}

\end{document}